\documentclass[11pt]{article}

\usepackage{amsmath,amssymb,amscd,amsthm,amsfonts}
\usepackage{graphicx,subfigure}
\usepackage{hyperref}
\usepackage{pstricks}
\usepackage{dsfont}
\usepackage{pifont}

\newtheorem{theorem}{Theorem}[section]
\newtheorem{proposition}[theorem]{Proposition}

\newtheorem*{theoremp}{Theorem}
\newtheorem{lemma}[theorem]{Lemma}

\newtheorem{corollary}[theorem]{Corollary}
\newtheorem{conjecture}[theorem]{Conjecture}

\newtheorem{definition}[theorem]{Definition}

\newcommand{\R}{\mathds{R}}

\newcommand{\Z}{\mathbb{Z}}

\newcommand{\ff}{\mathcal{F}}
\newcommand{\F}{\mathcal{F}}
\newcommand{\h}{{\mathds H}}
\newcommand{\tv}{{\mathds T}}
\newcommand{\cf}{{\cal F}}
\newcommand{\ck}{{\cal K}}

\newcommand{\ep}{\varepsilon}

\DeclareMathOperator{\conv}{conv}

\DeclareMathOperator{\vol}{vol}

\title{\bf Quantitative Tverberg, Helly, \& Carath\'eodory theorems}

\author{J.A. De Loera \and R. N. La Haye \and D. Rolnick \and P. Sober\'on}

\begin{document}

\maketitle


\section{Introduction}

Carath\'eodory's, Helly's, and Tverberg's theorems are among the most important theorems in convex geometry. Many generalizations and extensions, including colorful, fractional, and topological versions, have been developed and are a bounty for geometers. For a glimpse of the extensive literature see \cite{DGKsurvey63,Eckhoffsurvey93,Mbook,Wen1997,3nziegler} and the references therein. Our paper presents new quantitative versions of these classical theorems.  We distinguish between \emph{continuous quantitative} results, where we measure the size of our sets with a parameter, such as the volume or the diameter, which can vary continuously, and \emph{discrete quantitative} results, where we measure the size of our sets with an enumerative value, such as the number of lattice points they contain. The tables below summarize our results and prior work.

{\footnotesize

\begin{table}[h]\label{tablechida}
\begin{tabular}{|l|l|l|l|}
\hline \multicolumn{1}{|c|}{\em Monochromatic version} &\multicolumn{1}{|c|}{Carath\'eodory} & \multicolumn{1}{|c|}{Helly}  &\multicolumn{1}{|c|}{Tverberg}  \\
\hline Standard & \checkmark  & \checkmark  &  \checkmark+$(1.19-21)$ \\
\hline  Continuous Quantitative & \checkmark+ $(1.1-2,2.4-5)$& \checkmark+ $(1.4-5)$ & $(1.12)$ \\
\hline Discrete Quantitative & \checkmark+ $(1.3)$ & \checkmark+$(1.9)$ & $(1.17-18)$ \\
\hline   
\end{tabular}\vspace{12pt}
\begin{tabular}{|l|l|l|l|}
\hline \multicolumn{1}{|c|}{\em Colorful version} &\multicolumn{1}{|c|}{Carath\'eodory} & \multicolumn{1}{|c|}{Helly}  &\multicolumn{1}{|c|}{Tverberg}  \\
\hline Standard & \checkmark  & \checkmark & \checkmark \\
\hline  Continuous Quantitative & $(1.1-2, 2.4-5)$ & $(1.6)$ & $(1.14)$  \\
\hline Discrete Quantitative & $(1.3)$ & \checkmark+$(1.10-11)$  & ? \\
\hline 
\end{tabular}\vspace{12pt}
\caption{\emph{Prior and new results in quantitative combinatorial convexity.} The symbol $\checkmark $ means some prior result was known, $(\#)$ indicates the number of the theorem that is the first such result or a stronger version of prior results, and $?$ indicates an open problem.}\vspace{-12pt}
\end{table}

}

\subsection*{Classical versus quantitative theorems: history and results}
Before stating our main contributions, we recall the three classical theorems that are at the core of our work:

\begin{theoremp}[C.~Carath\'eodory 1911 \cite{originalCaratheodory}] 
\label{thm:Caratheodory}
Let $S$ be any subset of $\R^d$. Then each point in the convex hull of $S$ is a convex combination of at most $d+1$ points of $S$.
\end{theoremp}

\begin{theoremp}[E.~Helly, 1913 \cite{originalHelly}] 
\label{thm:Helly}
Let $\cf$ be a finite family of convex sets of $\R^d$. If $\bigcap \ck \neq
\emptyset$ for all $\ck \subset \cf$ of cardinality at most $d+1$, then
$\bigcap \cf \neq \emptyset$. 
\end{theoremp}

\begin{theoremp}[H.~Tverberg, 1966 \cite{Tverberg:1966tb}]
\label{thm:Tverberg}
 Let $a_{1},\ldots,a_{n}$ be points in $\R^{d}$.
If the number of points satisfies $n >(d+1)(m-1)$, then they can be partitioned into $m$ disjoint parts $A_{1},\ldots,A_{m}$ in such
a way that the $m$ convex hulls $\conv A_1, \ldots, \conv A_m$ have a point in common.
\end{theoremp}

The case of $m=2$ in Tverberg's theorem was proved in 1921 by J.~Radon \cite{originalRadon} and is often referred to as Radon's theorem or Radon's lemma. See \cite{Mbook} for an introduction to combinatorial convexity.

This paper provides several new quantitative versions of these three theorems where now the hypothesis and conclusion of theorems include measurable or enumerable information. Typical measurements involve the volume, the diameter, or the number of lattice points.

A key idea in our proofs of continuous quantitative results is showing a link to the efficient approximation of convex sets by polytopes. Convex body approximation is an active field that has seen great advances recently, which we apply in this article. We later state precisely the results we need, but recommend \cite{gruber1993aspects, bronstein2008approximation} for references on the subject.

On the other hand, for the proofs of the discrete quantitative theorems, we employ the fact that arguments work
more generally for \emph{restricted convexity} over discrete sets of $\R^d$. This means the sets we consider are the intersections of usual convex sets in $\R^d$ with a subset $S$ of $\R^d$ (e.g., $S=\Z^d$). For instance, we provide an enumerative generalization of a 1979 theorem by A.J.~Hoffman on how to compute the Helly number and a new notion of \emph{quantitative Helly number}. These two play a key role in our results for Helly and Tverberg theorems. Although we choose not to work in the most abstract and general setting possible, we note the convex hull operator in $\R^d$ equips $S$ with the structure of a general \emph{convexity space}. Convexity spaces are an axiomatic abstraction of usual convexity over $\R^d$ and many notions discussed 
here are valid even in that context. See \cite{AW2012,Kay:1971uf,queretaro,Doignon:1981fv,Jamison:1981wz, convexityspaces-vandevel} for more on this subject.

Finally, an important point we wish to stress is that we managed to present
an interconnected theory where our Carath\'eodory-type theorems
imply Helly-type results, and they in turn imply Tverberg-type statements.
For example, in Corollary \ref{corollary-simple-S-tverberg} we show that Tverberg numbers exist whenever Helly numbers exist, and Theorem \ref{thm:quantitative-disc-tverberg}.

The rest of the introduction lists our new theorems divided by type. In Section \ref{section-caratheodory} we give the proofs related to Carath\'eodory-type results, in Section \ref{section-helly} those related to Helly-type results, and finally in Section \ref{section-Tverberg} the proofs of Tverberg-type results.


\subsection*{Carath\'eodory-type contributions}

Carath\'eodory's theorem has interesting consequences and extensions (e.g., \cite{baranyonn-colorfulLP,Mbook}).
In 1914, the great geometer Steinitz improved the original proof by Carath\'eodory (which applied only to compact sets \cite{originalCaratheodory}) and at the same time he was the first to realize that this theorem has a nice version for points in the interior of a convex set:

\begin{theoremp}[E.~Steinitz, 1914 \cite{originalsteinitz}]
Consider $X \subset \R^d$ and $x$ a point in the interior 
of the convex hull of $S$. Then, $x$ belongs to the interior of the convex hull of a set of at most $2d$ points of $X$.
\end{theoremp}

A \emph{Carath\'eodory-type theorem} has a similar setup  where the points of the convex hull of a set $S$ can be expressed as convex combinations of a given number of generators with some additional conditions imposed. A monochromatic quantitative Carath\'eodory-type theorem was first proved by  B\'ar\'any, Katchalski, and Pach. These three mathematicians were the first to present quantitative theorems in combinatorial convexity.  We denote by $B_r(p) \subset \R^d$ the Euclidean ball of radius $r$ with center $p$.

\begin{theoremp}[I.~B\'ar\'any, M.~Katchalski, J.~Pach, 1982 \cite{baranykatchalskipach}]
There is a constant $r(d) \ge d^{-2d}$ such that the following statement holds.  For any set $X$ such that $B_1(0) \subset \conv X$, there is a subset $X' \subset X$ of at most $2d$ points that satisfies $B_{r(d)}(0) \subset \conv X'$.
\end{theoremp}

B\'ar\'any et al.~used this theorem  as a key lemma to prove their main quantitative results. We follow the same idea, but instead use the following colorful version of Steinitz' theorem:

\begin{theorem}[Colorful quantitative Steinitz with containment of small balls]\label{theorem-colored-steinitz}
Let $r(d) \ge d^{-2d-2}$ and $X_1, X_2, \ldots, X_{2d}$ be sets in $\R^d$ such that $B_1 (0) \subset \conv(X_i)$ for all $i$.  Then, we can choose $x_1 \in X_1, x_2 \in X_2, \ldots, x_{2d} \in X_{2d}$ so that
\[
B_{r(d)}(0) \subset \conv\{x_1, x_2, \ldots, x_{2d}\}.
\]
\end{theorem}

The reason for this result to be called ``colorful'' is that it has the following interpretation.  If every $X_i$ is painted with a different color, the theorem states that \textit{if the convex hull of every monochromatic set contains $B_1(0)$, then there is a colorful set whose convex hull contains $B_{r(d)}(0)$}.  This follows the lines of B\'ar\'any's generalization of Carath\'eodory's theorem \cite{baranys-caratheodory}: \textit{If $V_1,\cdots,V_{d+1}\subseteq{\R}^d$ and $p\in\bigcap_{i=i}^{d+1}\conv (\,V_i)$, then there exist elements $v_i\in V_i$, $1\leq i\leq d+1$, such that $p\in\conv\{v_1,\cdots,v_{d+1}\}$}.  A colorful version of Steinitz' original (non-quantitative) theorem was also noted by Jer\'onimo-Castro but never published \cite{jeronimo-colorful-communication}.

We also obtain a colorful version of Steinitz when we wish to optimize over the volume of $\conv\{x_1, \ldots, x_n\}$.  The constant $n(d,\ep)\sim\left( \frac{cd}{\ep}\right)^{(d-1)/2}$ for some absolute constant $c$ appearing below is related to how efficiently one can approximate convex sets by polytopes with few vertices. Definition \ref{definition-inscribed-volume} gives the explicit value of $n(d,\ep)$; which gives the correct bound for the following result up to a multiplicative factor of $d$.

\begin{theorem}[Colorful quantitative Steinitz with volume]\label{theorem-thrifty-steinitz}	For $d$ a positive integer and $\ep>0$ a constant, take $n=n(d,\ep)$ as in Definition \ref{definition-inscribed-volume}.  Then, the following property holds: If $X_1, X_2, \ldots, X_{nd}$ are sets in $\R^d$ and $K \subset \bigcap_{i=1}^{nd} \conv(X_i)$ is a convex set of volume $1$, we can choose $x_1 \in X_1, x_2 \in X_2, \ldots, x_{nd} \in X_{nd}$ so that
\[
\vol(\conv\{x_1, x_2, \ldots, x_{nd}\})\ge 1-\ep.
\]
Moreover, $n(d,\ep)$ is also a lower bound for the number of sets needed in this theorem.
\end{theorem}

For the applications of colorful Steinitz theorems, we need to optimize over slightly different parameters than the volume. These variations are Theorem \ref{theroem-steinitz-banach-mazur} and Proposition \ref{corollary-steinitz-for-balls} in Section \ref{section-caratheodory}; they follow the same scheme as the theorem above. Each is based on a constant related to different types of approximations of convex sets by polytopes. The continuous quantitative versions of Steinitz' theorem are at the core of our proofs for continuous quantitative versions of Helly's and Tverberg's theorems.

We next consider a discrete quantitative analogue of Carath\'eodory's theorem. How to generalize the theorem depends on whether we aim to quantify the size of the set contained in the convex hull in a discrete way or whether we want to force the input parameters to be integral or otherwise discrete. We consider the former type of generalization, for which we obtain the following result using standard methods.

\begin{theorem}[Colorful discrete quantitative Carath\'eodory]\label{theorem-quantitative-discrete-caratheodory}
	Let $K$ be a subset of $n \ge 2$ points in $\R^d$, and $\operatorname{ex}(K)$ be the number of extreme points of $K$.  
	If $n=\operatorname{ex}(K)$ and $X_1, X_2, \ldots, X_{nd}$ are sets whose convex hulls contain $K$, then we can find $x_1 \in X_1, \ldots, x_{nd} \in X_{nd}$ such that
	\[
	K \subset \conv\{x_1, \ldots, x_{nd}\}.
	\]
	Moreover, the number of sets is optimal for the conclusion to hold.
\end{theorem}

We believe this result may already be known, but we have not found references to it.  A proof is contained in Section \ref{section-caratheodory}. We will make use of this result in our proof of Theorem \ref{thm:quantitative-disc-tverberg}.

\subsection*{Helly-type contributions}

Helly's theorem and its numerous extensions are of central  importance in discrete and computational geometry (see
\cite{DGKsurvey63,Eckhoffsurvey93,Wen1997}). Helly himself understood immediately that his theorem had many variations, and was, for instance, the first to prove a topological version of his own theorem \cite{hellyagain}.  A \emph{Helly-type} property $P$ is a property for which there is a number $\mu$ such that the following statement holds.  \textit{If $\ff$ is a finite family of objects such that every subfamily with $\mu$ elements satisfies $P$, then $\ff$ satisfies $P$}.  A vague way to summarize some of the results below is that ``\textit{the intersection has a large volume}'' is a Helly-type property for convex sets.

To our knowledge, the first family of quantitative Helly-type theorems was made explicit by B\'ar\'any, Katchalski, and Pach in  \cite{baranykatchalskipach}. They obtained extensions of the classic Helly and Steinitz theorems for convex sets with a volumetric constraint.


\begin{theoremp}[B\'ar\'any, Katchalski, Pach, 1982 \cite{baranykatchalskipach}]
Let $\ff$ be a finite family of convex sets such that for any subfamily $\ff'$ of at most $2d$ sets,
\[
\vol\left( \cap \ff'\right) \ge 1 .
\]
Then,
\[
\vol\left( \cap \ff \right) \ge d^{-2d^2}.
\]
\end{theoremp}

This has recently been improved by Nasz\'odi to conclude $\vol (\cap \ff) \ge d^{-cd}$ for some absolute constant $c$ \cite{nazo15}.  The size of the subfamilies one must check cannot be improved over $2d$, as is noted in \cite{baranykatchalskipach}. In order to see this, let $\ff$ be the family of $2d$ halfspaces defining the facets of an arbitrarily small hypercube.  Any $2d-1$ define an unbounded polyhedron with non-empty interior, showing the optimality of their result.

In Section \ref{section-helly}, we show that it is possible to obtain better approximations of the volume of the intersection, namely $\vol\left( \cap \ff \right) \ge 1-\ep$, if one is willing to check for subfamilies $\ff'$ of larger size. This answers a question raised by Kalai and Linial during an Oberwolfach meeting in February 2015. The quantity $n^*(d, \ep)$ is defined properly in Definition \ref{definition-for-helly}; its asymptotic growth is similar to that of $n(d,\ep)$:

\begin{theorem}[Continuous quantitative Helly with volume]
\label{theorem-helly-for-volume}
Let $n=n^*(d, \ep)$ as in Definition \ref{definition-for-helly}.  Let $\ff$ be a finite family of convex sets such that for any subfamily $\ff'$ of at most $nd$ sets, 
\[
\vol\left( \cap \ff'\right) \ge 1 .
\]
Then,
\[
\vol\left( \cap \ff \right) \ge (1+\ep)^{-1}.
\]
Moreover, $n^*(d,\ep)$ is a lower bound for the size of the subfamilies $\ff'$ that we need to check.
\end{theorem}

We also present a quantitative version with diameter guarantees. The constant $n^{\operatorname{diam}}(d, \ep)$ is explained in Proposition \ref{definition-diameter}.  This comes from approximating convex sets with polytopes of few facets and bounded diameter.  In that proposition we show that the number of facets needed for efficient approximations can be bounded only in terms of the dimension and a the relative error on the diameter.  It is known that in order to approximate the unit sphere within distance $\ep$ in the  Hausdorff metric with a polytope, we require $\Omega(\ep^{-(d-1)/2})$ facets \cite{bronstein2008approximation}.  Thus $n^{\operatorname{diam}}(d, \ep)= \Omega(\ep^{-(d-1)/2})$.

\begin{theorem}[Continuous quantitative Helly with diameter]
\label{theorem-helly-for-diameter}
	Let $n=n^{\operatorname{diam}}(d, \ep)$ as in Propositionv \ref{definition-diameter}.  Let $\ff$ be a finite family of convex sets such that for any subfamily $\ff'$ of at most $nd$ sets, 
\[
\operatorname{diam}\left( \cap \ff'\right) \ge 1	.
\]
Then,
\[
\operatorname{diam}\left( \cap \ff \right) \ge (1+\ep)^{-1}.
\]
Moreover, $n$ is a lower bound for the size of the subfamilies $\ff'$ that we need to check.
\end{theorem}

The lower bounds presented in Theorem \ref{theorem-helly-for-volume} and Theorem \ref{theorem-helly-for-diameter} show that it is impossible to conclude $\vol(\cap \ff) \ge 1$ or $\operatorname{diam}(\cap\ff)\ge 1$, respectively, regardless of the size of the subfamilies we are willing to check.  This is a remarkable difference between the continuous and discrete quantitative Helly-type theorems.  In our final continuous quantitative Helly result, we generalize Theorem \ref{theorem-helly-for-volume} to the colorful setting.

\begin{theorem}[Colorful continuous quantitative Helly with volume.]
\label{thm:hcolorcont}
For any positive integer $d$ and $\ep >0$, there exists $n=n^h(d, \ep)$ such that the following holds.  Let $\ff_1, \ldots, \ff_n$ be $n$ finite families of convex sets such that for every choice $K_1 \in \ff_1, \ldots, K_n \in \ff_n$ we have
\[
\vol{\left( \bigcap_{i=1}^n K_i\right)} \ge 1.
\]
Then, there is an index $i$ such that
\[
\vol\left( \bigcap \ff_i \right) \ge 1-\ep.
\]
\end{theorem}

Before we state our discrete quantitative versions of Helly's theorem, we introduce an extension of the usual Helly number.

\begin{definition}
Given a set $S \subset \R^d$, the \emph{$S$-Helly number} $\h_S$ (if it exists) is the smallest positive integer with the following property. Suppose that $\mathcal F$ is a finite family of convex sets in $\R^d$, and that $\bigcap\mathcal G$ intersects $S$ for every subfamily $\mathcal G$ of $\mathcal F$ having at most $\h_S$ members. Then $\bigcap\mathcal F$ intersects $S$.
\end{definition}

 Note that the $\R^d$-Helly number is the usual $d+1$ of the standard Helly's theorem. Recall that a set $S$ is \emph{discrete} if every point $x\in S$ has a neighborhood such that $x$ is the only point of $S$ within it. A simple example is the lattice $\Z^d$. When $S$ is a discrete set, such as a lattice, the intersections are countable; thus we are
 able to quantify by counting points.
For a lattice $L$, Doignon was the first to calculate the $L$-Helly number, which has since been much studied by researchers in optimization (see e.g., \cite{Bell:1977tm,Sca1977,Hoffman:1979ix,clarkson}).

\begin{theoremp}[J.-P. Doignon, 1973 \cite{Doi1973}] 
Let $L$ be a rank-$d$ lattice inside $\R^d$. Then, $\h_L$ exists and is at most $2^d$.
\end{theoremp}





Doignon's theorem is just one of many results about $S$-Helly numbers. For instance, we know $\h_{\Z^{a} \times \R^{b}}=(b+1)2^{a}$ (see \cite{AW2012}). Most relevant for us are the results in \cite{queretaro} which generalized Doignon's theorem for discrete sets that are not lattices, effectively bounding $\h_S$ in several new situations.

\begin{theoremp}[J.A. De Loera et al. \cite{queretaro}] 
Let $L$ be a lattice in $\R^d$ and let $L_1,\dots,L_m$ be $m$ sublattices of $L$.  Let $R_m$ be the Ramsey number $R(3,3,\dots,3)$, i.e., the minimum number of vertices needed to guarantee the existence of a monochromatic triangle in any edge-coloring, using $m$ colors, of the complete graph $K_{R_m}$. Then the set $S = L \setminus (L_1\cup\dots\cup L_k)$ satisfies $\h_S\le (R_m-1)2^d$.

\end{theoremp}

 
In \cite{ipcoversion}, Aliev, De Loera, and Louveaux first showed an integer quantitative Helly-type theorem over $\Z^d$, generalizing Doignon's theorem (the bounds for the $\Z^d$-Helly number were later improved in \cite{alievetal}). Our Theorem \ref{theorem-helly-for-volume} matches closely the structure of the following result:
 
\begin{theoremp}[I. Aliev et al. \cite{ipcoversion}, 2014] 
Let $d,k$ be positive integers and $L \subset \R^d$ be a lattice of rank $d$. Then, there is a universal constant $c(d,k) \leq  \lceil 2(k+1)/3\rceil 2^d-2\lceil 2(k+1)/3\rceil+2$ such that the following property holds. For any collection $(X_i)_{i \in \Lambda}$ of closed convex sets in $\R^d$, where at least one of the sets is bounded, and exactly $k$ points of $L$ are in $\bigcap_{i \in \Lambda} X_i$, there is a subcollection of size at most $c(d,k)$ with the same $k$ lattice points in its intersection.
\end{theoremp}

We now  present a generalization of the preceding theorems. We use the following definition based on \cite{alievetal,ipcoversion}. The condition that $S$ must be discrete is necessary if the following definition is to make sense for $k>1$.



\begin{definition}
Given a discrete set $S \subset \R^d$, the \emph{quantitative $S$-Helly number} $\h_S(k)$ (if it exists) is the smallest positive integer with the following property. Suppose that $\mathcal F$ is a finite family of convex sets in $\R^d$, and that $\bigcap\mathcal G$ intersects $S$ in at least $k$ points for every subfamily $\mathcal G$ of $\mathcal F$ having at least $\h_S(k)$ members. Then $\bigcap\mathcal F$ intersects $S$ in at least $k$ points
\end{definition}

\begin{theorem}[Discrete quantitative Helly for differences of lattices] \label{quantitative-discrete-doignon} 
Let $L$ be a lattice in $\R^d$ and let $L_1,\dots,L_m$ be $m$ sublattices of $L$.
 Let $S=L \setminus (L_1\cup\dots\cup L_m)$. Then the quantitative $S$-Helly number $\h_S (k)$ exists and is bounded above by $\left(2^{m+1}k+1\right)^r$, where $r=\textnormal{rank}(L)$. 
\end{theorem}

Theorem \ref{quantitative-discrete-doignon} can be made into a colorful version. In fact, more can be said. As long as the set $S$ is discrete and has a finite quantitative $S$-Helly number $\h_S (k)$ then there will be colorful version too.  The conditions needed to be able to derive colorful Helly-type theorems have been used by several authors e.g., \cite{AW2012,baranymatousek}, and recently summarized in \cite{queretaro}.
 
As noted above, one needs the fact that the property \emph{``having at least $k$ points of $S$''}
has a finite $S$-Helly number. Second,
the property of having at least $k$ points of $S$ is 
 \emph{monotone} in the sense that if  $K\subset K'$ and $K$ has at least $k$ points from $S$, then this implies that $K'$ has also at least $k$ points of $S$ within. Finally, the property of having at least $k$ points from $S$ is \emph{orderable}, because  for any finite family $\F$ of convex sets there is a direction $v$ such that:
\begin{enumerate}
\item For every $K\in\F$ with $|K \cap S|\geq k$, there is a containment-minimal $v$-semispace (i.e. a half-space of the form $\{x:v^T x\ge 0\}$) $H$ such that $|K\cap H|\geq k$.
\item There is a unique containment-minimal $K'\subset K\cap H$ with $|K' \cap S|\geq k$.
\end{enumerate}

In our case, the work presented in \cite{queretaro} shows that
every monotone and orderable property with a well-defined Helly number must \emph{colorable}. 
This together with Theorem \ref{quantitative-discrete-doignon} yields the following:

\begin{theorem}Let $S$ be a discrete set in $\R^d$ with finite quantitative $S$-Helly number $N=\h_S(k)$.
 If $\F_1,\dots\F_N$ are finite families of closed convex sets (we think of each being a different color classes) such that $|\bigcap {\cal G} \cap S|\geq k$ for every rainbow subfamily ${\cal G}$ (i.e. a family with $|{\cal G} \cap\F_i|=1$ for every $i$), then $|\bigcap\F_i \cap S|\geq k$ for some color family $\F_i$.
\end{theorem}

\begin{corollary}[Colorful quantitative Helly for differences of lattices] \label{aquantitative-discrete-doignon} 
Let $L$ be a lattice in $\R^d$ and let $L_1,\dots,L_m$ be $m$ sublattices of $L$.
 Let $S=L \setminus (L_1\cup\dots\cup L_m)$.  Let $N=\left(2^{m+1}k+1\right)^r$ where $r=\textnormal{rank}(L)$ and $\F_1, \ldots, \F_N$ be finite families of closed convex sets so that $|\bigcap {\cal G} \cap S|\geq k$ for every rainbow subfamily ${\cal G}$.  then, there is an $i$ such that $|\bigcap\F_i \cap S|\geq k$.
\end{corollary}

\subsection*{Tverberg-type contributions}

Helge Tverberg proved his classic theorem in 1966  \cite{Tverberg:1966tb}. Later in 1981 he published another proof  \cite{Tverberg1981}, and simpler proofs have since appeared in  \cite{baranyonn-colorfulLP}, \cite{Sarkaria:1992vt},  and \cite{Roudneff}.  Chapter \cite[\S8.3]{Mbook} and the expository article \cite{3nziegler} can give the reader a sense of the abundance of work surrounding this lovely theorem. Here we present the first quantitative versions, in both continuous and discrete settings.

%

First,  we prove a version of Tverberg's theorem where each convex hull must contain a Euclidean ball of given radius.  In other words, we measure the ``size'' of $\cap_{i=1}^m \conv A_i$ by the inradius. Our proof combines Tverberg's theorem with our two versions of quantitative Steinitz' theorem for balls, Theorem \ref{theorem-colored-steinitz} and Proposition \ref{corollary-steinitz-for-balls}. The constant $n^{\operatorname{bm}}(d,\ep) \sim e^{-d/2}$ will be made explicit in Definition \ref{banach-mazur-value}. 

Note that, unlike the classical Tverberg theorem, some conditions must be imposed on the set of points to be able to obtain such a result.  For instance, regardless of how many points we start with, if they are all close enough to some flat of positive co-dimension, then all hopes of a continuous quantitative version of Tverberg's theorem quickly vanish.
In order to avoid the degenerate cases, we make the natural assumption that the set of points is ``thick enough''.

\begin{theorem}[Continuous quantitative Tverberg] \label{theorem-tverberg-volume}
Let $n=(2dm-1)(d+1)+1$ and $T_1, T_2, \ldots, T_n$ be subsets of $\R^d$ such that the convex hull of each $T_i$ contains an Euclidean ball of radius one, $B_1(c_i)$.  
Then, we can choose points $t_1 \in T_1, t_2 \in T_2, \ldots, t_n \in T_n$ and a partition of $\{ t_1, t_2, \ldots, t_n\}$ into $m$ sets $A_1, A_2, \ldots, A_m$ such that the intersection
\[
\bigcap_{i=1}^m \conv{A_i}
\]
contains a ball of radius $d^{-2(d+1)}$.

Moreover, if  we instead take $n'= n^{\operatorname{bm}}(d, \ep)$ and we let
\[
n= (m\cdot [(n'-1)d+1]-1)(d+1)+1,
\]
then we can guarantee that $\bigcap_{i=1}^m \conv A_i$ contains a ball or radius $(1+\ep)^{-1}$.
\end{theorem}

As with Helly's and Carath\'eodory's theorems, there are colorful versions of Tverberg's theorem.  In this case, the aim is to impose additional combinatorial conditions on the resulting partition of points, while guaranteeing the existence of a partition where the convex hulls of the parts intersect. Now that the conjectured topological versions of Tverberg's theorem have been proven false \cite{frick15}, the following conjecture by B\'ar\'any and Larman is arguably the most important open problem surrounding Tverberg's theorem.

\begin{conjecture}[B\'ar\'any and Larman, 1992 \cite{Barany:1992tx}]\label{conjecture-colorful-Tverberg}
Let $F_1, F_2, \ldots, F_{d+1} \subset \R^d$ be sets of $m$ points each, considered as color classes.  Then, there is a colorful partition of them into sets $A_1, \ldots, A_{m}$ whose convex hulls intersect.
\end{conjecture}

By a colorful partition $A_1, \ldots, A_m$ we mean that it satisfies $|A_i \cap F_j| = 1$ for all $i,j$. In presenting the conjecture, B\'ar\'any and Larman showed that it holds for $d=2$ and any $m$, and included a proof by L\'ov\'asz for $m=2$ and any $d$.  Recently, Blagojevi\'c, Matschke, and Ziegler \cite{blago3, bmz-optimal} showed that it is also true for the case when $m+1$ is a prime number and any $d$. The reason for these conditions on the parameters of the problem is that their method of proof uses topological machinery requiring these assumptions. However, their result shows that if we allow each $F_i$ to have $2m-1$ points instead of $m$, we can find $m$ pairwise disjoint colorful sets whose convex hulls intersect, without any conditions on $m$.  For variations of conjecture \ref{conjecture-colorful-Tverberg} that do imply Tverberg's theorem, see \cite{blago3, bmz-optimal, Soberon:2013fr}.


Combining results of Blagojevi\'c, Matschke, and Ziegler with our two colorful Steinitz theorems, we can obtain volumetric versions of these results similar to Theorem \ref{theorem-tverberg-volume}.  In order to obtain a ball in the intersection, this time we must also allow each $A_i$ to have more points of each color class.  For an integer $q$, let $\lceil q \rceil_p$ the smallest prime which is greater than or equal to $q$.  Then

\begin{theorem}[Colorful continuous quantitative Tverberg] \label{theorem-colorful-tverberg-volume}

Let $n=\lceil 2md+1\rceil_p -1$ and $F_1, F_2, \ldots, F_{d+1}$ be families of $n$ sets of points of $\R^d$ each.  We consider the families $F_i = \{T_{i,j}: 1 \le j \le n\}$ as the color classes.  Suppose that $\conv (T_{i,j})$ contains a ball of radius $1$ for all $i,j$.  Then, there is a choice of points $t_{i,j} \in T_{i,j}$ and a partition of the resulting set into $m$ parts $A_1, \ldots, A_{m}$ such that each $A_i$ contains at most $2d$ points of each color class and $\bigcap_{i=1}^m \conv (A_i)$ contains a ball of radius $d^{-2(d+1)}$.
	
In addition, if we take instead $n'= n^{\operatorname{bm}}(d,\ep)$ and 
\[
n=\lceil m\cdot ((n'-1)d+1)+1\rceil_p -1,
\]
and allow each $A_i$ to have $(n'-1)d+1$ points of each color, then in the conclusion we can guarantee that $\bigcap_{i=1}^m \conv (A_i)$ contains a ball of radius $(1+\ep)^{-1}$.
\end{theorem}

The reason why we require the use of $\lceil q \rceil_p$ is the conditions for the known cases of Conjecture \ref{conjecture-colorful-Tverberg}.  If Conjecture \ref{conjecture-colorful-Tverberg} were proved, we could use $2dm$ sets in each color class instead. However, since the prime number theorem implies $\lim_{q \to \infty} \frac{\lceil q \rceil_p}{q} =1$ and in the small cases we have $\lceil q \rceil_p < 2q$, the result above is almost as good. We should note that the ``optimal colorful Tverberg'' by Blagojevi\'c, Matschke, and Ziegler \cite[Theorem 2.1]{bmz-optimal} also admits a volumetric version as above, with essentially the same proof. 

If all $T_{i,j}$ are equal to $B_1(0)$, the need to allow each $A_i$ to have more points from each color class becomes apparent from the results of inaproximability of the sphere by polytopes with few vertices \cite{bronstein2008approximation}.  The condition we have is saying that the number of points from $F_j$ in $A_i$ should not exceed $\frac{1}{m}|F_j|$.  We know that a subset of $B_1(0)$ that contains $B_{1-\ep}(0)$ should have at least $n^{\operatorname{bm}}(d,\ep)$ points, showing that the number of points we are allowing to take from each color class is optimal up to a multiplicative factor of $\sim d^2$.

To explain our next contributions we begin by remarking that
traditionally Tverberg's theorem considers intersections over $\R^d$.
Here we will be interested in a Tverberg number,
where the points are in $S \subset \R^d$ and intersections of the convex hulls of the partition sets are required to have non-empty intersection with $S$. More precisely, we make the following definition.

\begin{definition}
Given a set $S\subset \R^d$, the \emph{$S$-Tverberg number} $\tv_{S}(m)$ (if it exists) is the smallest positive integer such that among any $\tv_{S}(m)$ distinct points in $S\subseteq\R^d$, there is a partition of them into $m$ sets $A_1,A_2,\dots,A_m$ such that the intersection of their convex hulls contains some point of $S$.
\end{definition}

For example, when $S=\Z^d$, we wish to have 
enough lattice points to
be partitioned into $m$ sets whose convex hulls' intersection contains
a lattice point. It was previously known that $2^d (m-1) < \tv_{\Z^d}(m) \le (m-1)(d+1)2^d - d - 2$. These bounds are mentioned by 
Eckhoff \cite{Eckhoff:2000jw}.  The upper bound follows by combining a theorem of Jamison for general convexity spaces \cite{Jamison:1981wz} with \cite{Doi1973}. We improved this bound
in this paper (see Corollary \ref{corollary:integer-tverberg} below).

In quantitative discrete theorems we wish to enumerate points. Counting points in a lattice is natural, but not in a dense set such as $S=({\mathbb Q}[\sqrt{2}])^d$. Here we go beyond lattices and consider more sophisticated discrete subsets $S$ of $\R^d$. We begin with the following definition.

\begin{definition}
Given a discrete subset $S$ of $\R^d$, the \emph{quantitative $S$-Tverberg number} $\tv_{S}(m,k)$ (if it exists) is the smallest positive integer such that among any $\tv_{S}(m,k)$ distinct points in $S\subseteq\R^d$, there is a partition of them into $m$ sets $A_1,A_2,\dots,A_m$ such that the intersection of their convex hulls contains at least $k$ points of $S$.
\end{definition}

Note that the definition of the \emph{quantitative} $S$-Tverberg number makes sense only when $S$ is discrete.
Now we present the first quantitative discrete Tverberg theorem.

\begin{theorem}[Discrete quantitative Tverberg] \label{thm:quantitative-disc-tverberg}
Let $S\subseteq \R^d$  with finite quantitative Helly number $\h_S(k)$. Let $m,k$ be integers with $m,k\ge 1$. Then, we have
\[
\tv_S(m,k)\leq \h_S(k)(m-1)kd+k.
\]
\end{theorem}

This theorem produces many fascinating corollaries, 
of which we list only a few that follow directly. 
First, using the quantitative Helly theorem for $\Z^d$ in \cite{alievetal,ipcoversion}, we obtain the following.

\begin{corollary}[Discrete quantitative Tverberg over $\Z^d$]  Set $c(d,k)=\lceil 2(k+1)/3\rceil 2^d-2\lceil 2(k+1)/3\rceil+2$. The quantitative Tverberg number of the
integer lattice $\Z^d$ exists and is bounded by
\[
\tv_{\Z^d}(m,k)\leq c(d,k) (m-1)kd+k.
\]
Therefore, any set of at least $\tv_{Z^d}(m,k)$ many integer  lattice points can be partitioned into $m$ disjoint subsets such that their convex hulls intersect in at least $k$ lattice points.
\end{corollary}

{\bf Remark:} As can be seen from the proof, the assumption that $S$ be discrete is not necessary in Theorem \ref{thm:quantitative-disc-tverberg}. However, when $k>1$, Theorem \ref{thm:quantitative-disc-tverberg} is most interesting in the case that $S$ is discrete, since we wish to count points of the intersection. However, in the case of $k=1$, we have no enumeration and care only about a non-empty intersection over $S$. We simply consider Tverberg's theorem with points of $S\subset \R^d$ and the $S$-Tverberg number $\tv_S(m)$. Our next corollary therefore holds for subsets $S$ of $\R^d$ as long as they have a Helly number.

\begin{corollary}[$S$-Tverberg number exists when the $S$-Helly number exists]\label{corollary-simple-S-tverberg}
Suppose that $S\subseteq\R^d$ is such that $\h_S$ exists. (In particular, $S$ need not be discrete.) Then, the $S$-Tverberg number exists too and satisfies
\[
\tv_S(m)\leq(m-1)d\cdot\h_S+1.
\]
\end{corollary}

The next corollary uses the work on $S$-Helly
numbers presented in \cite{queretaro} and in \cite{AW2012}. In \cite{queretaro}, the authors presented many new bounds for Helly numbers of interesting subsets of $\R^d$. 

\begin{corollary}[$S$-Tverberg number for interesting families]
From Corollary \ref{corollary-simple-S-tverberg}, the following Tverberg numbers $\tv_S(m)$
exist and are bounded as stated in the following situations:
\begin{enumerate}
\item When $S=\Z^{d-a}\times\R^a$, we have $\tv_S(m)\leq (m-1)d (2^{d-a}(a+1))+1.$ 

\item Let $L',L''$ be sublattices of a lattice $L \subset\R^d$. 
Then, if $S=L \setminus (L' \cup L'')$, the Tverberg number satisfies
$\tv_S(m)\leq 6(m-1)d 2^d+1$.

\item If $S$ is an additive subgroup $S\subset\R^d$ with closure $\Z^{d-a}\times\R^a$, then  we have $\tv_S(m)\le (m-1)d \max(2^{d-a}(a+1),2^{d-1}+2)+1$. More strongly, if $S$ is also a ${\mathbb Q}$-module, the bound can be improved to $\tv_S(m)\le 2(m-1)d^2+1$.

\end{enumerate}
\end{corollary}

There is the very important case of $S=\Z^d$ that we highlight.
 Our Theorem \ref{thm:quantitative-disc-tverberg} allows us to improve the priorly known upper bound slightly from $O(m(d+1)2^d)$ to $O(md2^d)$ (see \cite{Eckhoff:2000jw,onn+radon} on prior results, in particular Onn's work on the case of Radon partitions ($m=2$)).

\begin{corollary}[Improvements on integer Tverberg]
\label{corollary:integer-tverberg}
Setting $S=\Z^d$, we obtain the following bound on the Tverberg number:
\[
\tv_{\Z^d}(m) \le (m-1)d 2^d +1.
\]
\end{corollary}


Finally, we believe that a discrete colorful quantitative Tverberg should be true too, in the same sense as Conjecture \ref{conjecture-colorful-Tverberg}.  Namely, we propose the following.

\begin{conjecture}
Let $S \subset \R^d$ be a set such that the Helly number $\h_S(k)$ is finite for all $k$.  Then, for any $m,k$ there are integers $m_1$ and $m_2$ such that the following statement holds.

Given $m_1$ families $F_1, F_2, \ldots, F_{m_1}$ families of $m_2$ points of $S$ each, considered as color classes, there are $m$ pairwise disjoint colorful sets $A_1, A_2, \ldots, A_m$ such that
\[
\bigcap_{i=1}^m \conv (A_i)
\]
contains at least $k$ points of $S$.
\end{conjecture}

Even in the case $k=1$, $S=\mathbb{Z}^d$, the question above remains interesting. As with Conjecture \ref{conjecture-colorful-Tverberg}, it would be desirable to have $m_2 = m$ in the cases where the above is true.
\section{Proofs of quantitative Carath\'eodory theorems}\label{section-caratheodory}


We prove only the colorful versions of our Carath\'eodory type theorems.  Given sets $X_1, \ldots, X_n$, considered as color classes, whose convex hulls contain a large set $K$, we want to make a colorful choice $x_1 \in X_1, \ldots, x_n \in X_n$ such that $\conv \{x_1, \ldots, x_n\}$ is also large. The monochromatic versions of the results below is simply the case when $X_1 = X_2 = \cdots = X_n$.

\subsection{Continuous quantitative Carath\'eodory}


There are two parameters we may seek to optimize.  One is the number $n$ of sets required to obtain some lower bound for the size of $\conv \{x_1, \ldots, x_n\}$.  The other is the size of $\conv \{x_1, \ldots, x_n\}$ assuming that the size of $K$ is $1$. We obtain a different result for each case.

The only existing quantitative result of this kind is a monochromatic quantitative version of Steinitz' theorem by B\'ar\'any, Katchalski, and Pach \cite{baranykatchalskipach}, quantifying the largest size of a ball centered at $0$ and contained in $K$, described in the introduction. The case when $X$ is the set of vertices of a regular octahedron centered at the origin shows that the number of points they use, $2d$, cannot be reduced. Here we show how adapting the proof of \cite{baranykatchalskipach} gives Theorem \ref{theorem-colored-steinitz}. The only extra ingredient needed is the ``very colorful  Carath\'eodory'' of Arocha, B\'ar\'any, Bracho, Fabila, and Montejano \cite[Theorem~2]{Arocha:2009ft}.



\begin{theoremp}[J.~Arocha et al. 2009 \cite{Arocha:2009ft}] \label{verycolorfulcaratheodory}
Let $X_1, X_2, \ldots, X_d \subset \R^d$ be sets, each of whose convex hulls contains $0$ and one additional point $p \in \R^d$.  Then, we can choose $x_1 \in X_1, \ldots, x_d \in X_d$ such that
	\[
	0 \in \conv\{x_1, x_2, \ldots, x_d, p\}.
	\] 
\end{theoremp}

\begin{proof}[Proof of Theorem \ref{theorem-colored-steinitz}] Our goal is to pick explicitly the $2d$ points $x_1,\dots,x_{2d}$. For this, let
 $P$ be a regular simplex of maximal volume contained in $B_1(0)$.  Note that $B_{1/d}(0) \subset P \subset B_1(0)$.  Since $P \subset X_i$ for an arbitrary $i$ and $P$ has $d+1$ vertices, by repeatedly applying Carath\'eodory's theorem we can see that there is a subset of $X_i$ of size at most $(d+1)^2$ whose convex hull contains $P$.  Thus, without loss of generality we may assume $|X_i| \le (d+1)^2$ and $B_{1/d}(0) \subset \conv(X_i)$ for all $i$.
	
  Given a collection of $d$ points, $x_1 \in X_1$, $x_2 \in X_2$, $\ldots$, $x_d \in X_d$, consider the convex (simplicial) cone spanned by them.  Let $C_1, C_2, \ldots, C_n$ be all possible cones generated this way.  The number of cones, $n$, is clearly bounded by
  \[
  n \le (d+1)^{2d}.
  \]
  \noindent \textbf{Claim.} The cones $C_1, C_2, \ldots, C_n$ cover $\R^d$.
  
  In order to prove the claim, it suffices to show that for each vector $v$ of norm at most $\frac1d$, there is a cone $C_i$ that contains it.  However, since $B_{1/d}(0) \subset X_i$ for all $i$ (in particular for the first $d$), we can apply the very colorful Carath\'eodory theorem above with the point in the convex hull being $v$ and the extra point being $0$.
  
  If we denote by $w_{d-1}$ the surface area of the unit sphere $S^{d-1}$, there must be one of the cones 
  $C_i$ which covers a surface area of at least $\frac1n w_{d-1}$.  We can assume without loss of generality that it is the first cone $C_1$.
 
  Let $a \in C_1$ be a unit vector whose minimal angle $\alpha$ with the facets of $C_1$ is maximal (i.e. we take the incenter of $C_1 \cap S^{d-1}$, with distance measured in the sphere).  Now we show that since the surface area of $C_1 \cap S^{d-1}$ is large, its inradius must also be large.  The argument we present is different from \cite{baranykatchalskipach}, giving a slightly worse constant.  Our final radius is $d^{-2d-2}$ as opposed to their $d^{-2d}$.
  
  For a facet $L_i$ of $C_1$, let $D_i$ be the set of points whose angle with $L_i$ is at most $\alpha$ and that lie on the same side of $L_i$ as $a$.  Note that $C_1$ has $d$ facets and 
  so $\cup_{i=1}^d D_i = C_1$.  The surface area of $S^{d-1} \cap D_i$ is clearly bounded by $\frac{\alpha}{2\pi}w_{d-1}$.  Thus
  \[
  \frac1n w_{d-1} \le \operatorname{Area}(S^{d-1} \cap C_1)  < \sum_{i=1}^d \operatorname{Area} (S^{d-1}\cap D_i)  \le \frac{d\alpha}{2\pi}w_{d-1},
  \]
which implies $\alpha > \frac{2\pi}{dn}$.  Now consider $a' =\frac{-1}d a$, the vector of norm $\frac1d$ in the direction opposite to $a$.  By applying the very colorful Carath\'eodory as before, 
we can choose now $x_{d+1} \in X_{d+1}, x_{d+2} \in X_{d+2}, \ldots, x_{2d} \in X_{2d}$ such that
  \[
  a' \in \conv\{0, x_{d+1}, x_{d+2}, \ldots, x_{2d}\}.
  \]
  Now consider the set $K = \{x \in \frac1d S^{d-1} : \angle (x, a) \le \alpha\}$.  Let $x_1 \in X_1, \ldots, x_d \in X_d$ be the $d$ points that generate $C_1$.  Notice that the cone with apex $a'$ 
  and base $K$ is contained in $\conv\{x_1, x_2, \ldots, x_{2d}\}$.  Finding the radius $r$ of the largest ball around $0$ that is contained in this new cone is easily reduced to a $2$-dimensional 
  problem, giving
  \[
  r = \frac{\tan \alpha}{2d} > \frac{\alpha}{2d} > \frac{\pi}{nd^2} > d^{-2d-2}.
  \]
  as we wanted.
\end{proof}

It seems more natural to optimize the size of $\conv\{x_1, \ldots, x_n\}$ instead of the integer $n$.  This optimization turns out to be closely related to finding efficient approximations of convex sets with polytopes.  This is a classic problem which has many other motivations, see \cite{gruber1993aspects,bronstein2008approximation, barvinok2014thrifty} for the state of the art and the history of this subject. In this paper we will need the following three important constants.

\begin{definition}\label{definition-inscribed-volume}
	Let $d$ be a positive integer and $\ep>0$.  We define $n(d,\ep)$ as the smallest integer such that, for any convex set $K \subset \R^d$ with positive volume, there is a polytope $P \subset K$ of at most $n(d,\ep)$ vertices such that
	\[
	\vol(P) \ge (1-\ep) \vol(K).
	\]
\end{definition}

\begin{definition}\label{banach-mazur-value}
	Let $d$ be a positive integer and $\ep>0$.  We define $n^{\operatorname{bm}}(d,\ep)$ as the smallest integer such that, for any centrally symmetric convex set $K \subset \R^d$ with positive volume, there is a polytope $P$ of at most $n^{\operatorname{bm}}(d,\ep)$ vertices and a linear transformation $\lambda: \R^d \to \R^d$ such that
	\[
	P \subset \lambda(K) \subset (1+\ep )P.
	\]
	In other words, there is always a polytope $P \subset K$ of fixed number of vertices which is within $\varepsilon$ of $K$ according to the Banach-Mazur distance.
\end{definition}

\begin{definition}\label{definition-for-helly}
Let $d$ be a positive integer and $\ep>0$.  We define $n^*(d,\ep)$ as the smallest integer such that, for any convex set $K \subset \R^d$ with positive volume, there is a polytope $P \supset K$ of at most $n^*(d,\ep)$ facets such that
\[
\vol(P) \le (1+\ep) \vol(K).
\]	
\end{definition}

The asymptotic behavior of $n(d, \ep)$ is known: 
\[\left(\frac{c_1d}{\ep}\right)^{(d-1)/2} \ge n(d,\ep) \ge \left(\frac{c_2d}{\ep}\right)^{(d-1)/2},\]
for absolute constants $c_1, c_2$.

This comes from approximating convex bodies with polytopes of few vertices via the Nikodym metric \cite[Section 4.2]{bronstein2008approximation}.  The lower and upper bounds can be found in \cite{gordon1995constructing} and \cite{gordon1997umbrellas}, respectively.

We will use these definitions to obtain both upper and lower bounds for our quantitative results. As shown below, $n(d,\ep)$ is precisely the number needed for a quantitative colorful Steinitz theorem with volume. The constant $n^{\operatorname{bm}}(d,\ep)$ will be needed to improve the quantitative Steinitz theorem if we are interested in determining the size of a set by the radius of the largest ball around the origin contained in it.  The bounds for $n^{\operatorname{bm}}(d, \ep)$ involve the condition of central symmetry as the Banach-Mazur distance is most natural when working with norms in Banach spaces. Recently, Barvinok has obtained very sharp estimates of $n^{\operatorname{bm}}(d,\ep)$ \cite{barvinok2014thrifty}, giving $n^{\operatorname{bm}}(d,\ep) \le \left(\frac{1}{\sqrt{\ep}}\ln \left(\frac{1}{\ep}\right)\right)^d$ if $d$ is large enough.  It should be noted that $n^{\operatorname{bm}}(d,\ep) = \Omega (\ep^{(-(d-1)/2})$ \cite{blaschke1923affine}.

Finally, the constant $n^*(d,\ep)$ is the key value for the continuous quantitative Helly theorems in Section \ref{section-helly}.  A result of Reisner, Sch\"ut and Werner shows that $n^*(d, \ep) \le 2 n(d, \ep)$ \cite[Section 5]{reisner2001dropping}.  Using the notation above, they actually find a polytope $P \subset K$ of few vertices such that $P$ has at least a $(1-\ep)$-fraction of the volume of $K$ and its polar $P^*$ has at most a $(1+\ep)$-fraction of the volume of $K^*$.

We are now ready to prove Theorem \ref{theorem-thrifty-steinitz}.


\begin{proof}[Proof of Theorem \ref{theorem-thrifty-steinitz}]
	Let $P \subset K$ be a polytope with $n=n(d,\ep)$ vertices such that $\vol(P) \ge (1- \ep) \vol(K)=(1-\ep)$.  We may assume without loss of generality that $0$ is in the interior of $P$.  Now label the vertices of $P$ as $y_1, y_2, \ldots, y_n$.  Using the very colorful Carath\'eodory theorem as in the previous proof, for a fixed $j$ we can find $x_{(j-1)d+1}{\in}X_{(j-1)d+1}$, $\ldots ,$ $x_{jd} \in X_{jd}$ such that
	\[
	y_j \in \conv\{0, x_{(j-1)d+1}, \ldots, x_{jd} \}.
	\]
	In order to finish the proof, it suffices to show that $0 \in \conv\{x_1, \ldots, x_{nd} \}$.  If this is not the case, then there is a hyperplane separating $0$ from $\conv\{x_1, \ldots, x_{nd} \}$.  We may assume that the hyperplane contains $0$ and leaves $x_1, \ldots, x_{nd}$ in the same closed halfspace.  Notice that then there would be a vertex of $y_j$ of $P$ in the other (open) halfspace, contradicting the fact that $y_j{\in}\conv\{0,x_1,\ldots,x_{nd}\}$. 
	
We now prove the near-optimality of our bound. Let $K$ be a convex set of volume $1$ such that for every polytope $P \subset K$ of at most $n-1$ vertices we have $\vol(P) < 1-\ep$.  Then, having $K = X_1 = X_2 = \cdots = X_{n-1}$ gives the desired counterexample, as any colorful choice of points has size $n-1$.
\end{proof}

If we want the subset to be close to $K$ in terms of the Banach-Mazur distance, we simply replace $n(d,\ep)$ by $n^{\operatorname{bm}}(d,\ep)$ and the same proof holds.

\begin{theorem}[Quantified colorful Steinitz for Banach-Mazur distance]\label{theroem-steinitz-banach-mazur}
Let $d$ be a positive integer and $\ep>0$ be a constant.  Set $n=n^{\operatorname{bm}}(d,\ep)$ and let $X_1, X_2, \ldots, X_{nd}$ be sets in $\R^d$ such that $K \subset \bigcap_{i=1}^{nd}\conv X_i$ is a centrally symmetric convex set with volume $1$.  Then, we can find $x_1 \in X_1, x_2 \in X_2, \ldots, x_{nd} \in X_{nd}$ and an affine transformation $\lambda:\R^d \to \R^d$ so that $\conv\{x_1, x_2, \ldots, x_{nd}\}$ contains a set $P$ with
\[
P \subset \lambda(K) \subset (1+\ep)P.
\]
Moreover, $n$ is a lower bound for the number of sets needed for this result to hold.
\end{theorem}
 
In particular, with the same ideas we get the following proposition, which improves the quantitative version of B\'ar\'any, Katchalski, and Pach when we want to optimize the radius of the balls contained in the set.  The number of sets we use is slightly improved by using the symmetries of the sphere.
\begin{proposition}\label{corollary-steinitz-for-balls}
Set $n=n^{\operatorname{bm}}(d,\ep)$ and let $X_1, X_2, \ldots, X_{(n-1)d+1}$ be sets in $\R^d$ such that $B_1 (0) \subset \bigcap_{i=1}^{(n-1)d+1} \conv X_i$. Then, we can choose $x_1 \in X_1, x_2 \in X_2, \ldots, x_{(n-1)d+1} \in X_{(n-1)d+1}$ so that
\[
B_{1/(1+\ep)}(0) \subset \conv\{x_1, x_2, \ldots, x_{(n-1)d+1}\}.
\]
\end{proposition}

\begin{proof}
We follow the same steps as in the proof of Theorem \ref{theorem-thrifty-steinitz}.  Once we have constructed a polytope $P \subset B_1(0)$ of $n$ vertices $y_1, y_2, \ldots, y_n$ such that $B_1(0) \subset (1+\ep)P$, we can rotate $P$ so that there is a point $x_{(n-1)d+1} \in X_{(n-1)d+1}$ such that
\[
y_{n} \in \conv\{0, x_{(n-1)d+1}\}.
\]
For the other $n-1$ vertices of $P$, we use the very colorful Carath\'eodory theorem as above.
\end{proof}

\subsection{Discrete quantitative Carath\'eodory}

Our discrete version of quantitative Carath\'eodory is less enigmatic. Indeed, the arguments above almost contain a proof of Theorem \ref{theorem-quantitative-discrete-caratheodory}. 


\begin{proof}[Proof of Theorem \ref{theorem-quantitative-discrete-caratheodory}]
After enumerating the extreme points of the set $K$ as $y_1, y_2, \ldots, y_n$, the proof follows the same argument as the proof of Theorem \ref{theorem-thrifty-steinitz}.  
The only difference is that one needs to assume that $0$ is in the relative interior of $K$.

In order to show the value $nd$ is optimal, consider a convex polytope $K'$ which has each $y_i$ in the relative interior of one of its facets, and such that the facets corresponding to $y_i$ and $y_j$ do not share vertices for all $i \neq j$.  Then, take $nd-1$ copies $X_1, X_2, \ldots, X_{nd-1}$ of $K'$.  Any colorful choice whose convex hull contains $K$ needs at least $d$ vertices for each extreme point of $K$, which is not possible.
\end{proof}
\section{Proofs of quantitative Helly theorems}\label{section-helly}

\subsection{Continuous quantitative Helly}


In this section we give proofs for our Helly-type results.  As mentioned in the introduction, the first quantitative Helly type theorem came from B\'ar\'any, Katchalski, and Pach's ground-breaking work \cite{baranykatchalskipach}.  They quantify the size of the intersection of the family in two ways, using volume and diameter respectively.  Their result with diameter is essentially equivalent to that with volume, though the final constant obtained is slightly different.



\begin{proof}[Proof of Theorem \ref{theorem-helly-for-volume}]
We may assume that $\cap \ff$ has non-empty interior. This was the first step in the original proof given in \cite{baranykatchalskipach}.  We may either use the same method or notice that if $n \ge 2$ we can actually use the``quantitative volume theorem'' of \cite[p.~109]{baranykatchalskipach} to obtain this. If $\cap \ff$ is not bounded, then it has infinite volume.  Moreover, we may assume that the sets in $\ff$ are closed halfspaces, or we could take the set of halfspaces containing $\cap \ff$ instead of $\ff$. Thus, it suffices to prove the following lemma.

\begin{lemma}
Let $\ff$ be a family of halfspaces such that $\cap \ff$ has volume $1$ and contains the origin in its interior. Take $n=n^*(d,\ep)$, the constant of Definition \ref{definition-for-helly}. Then, there is a subfamily $\ff' \subset \ff$ of at most $nd$ elements such that $\vol(\cap \ff') \le 1+\ep$.
\end{lemma}

To prove the lemma, consider a polytope $K$ of $n$ facets containing $P=\cap \ff$ such that $\vol(K) \le 1+\ep$. Such polytope exists by the definition of $n^*(d,\ep)$. After taking polars, we have $K^* \subset P^*$, and $K^*$ is a convex polytope with $n$ vertices.  Let $\ff^*$ be the family of polars of the elements in $\ff$.  Note that $\conv(\ff^*) = P^*$.  Thus, we can apply Theorem \ref{theorem-quantitative-discrete-caratheodory} with $X_1 =X_2 = \cdots = X_{nd} = \ff^*$ and find a subset $\ff' \subset \ff$ of at most $nd$ elements such that $K^* \subset \conv [(\ff')^*]$.  Then
	\[
	P \subset \cap \ff' \subset K,
	\]
	giving the desired result.
	
In order to prove optimality, let $K$ be a convex polytope of volume $1$ such that any polytope $P \supset K$ with at most $n-1$ facets has volume greater than $1+\ep$; this exists by the definition of $n^*(d,\ep)$.  Let $\mathcal{F}$ be the set of closed halfspaces that contain $K$ and define a facet of $K$.  Clearly, there is a $\delta >0$ such that the intersection of every $n-1$ elements of $\ff$ has volume at least $1+\ep+\delta$, but the intersection $\cap \ff$ is of volume $1$.
\end{proof}

Once we have constructed the polytope $K$, we can also finish the proof with the following folklore lemma that follows from Helly's theorem.

\begin{lemma}
Let $\mathcal{F}$ be a finite family of convex sets and $H$ a closed halfspace such that $\cap \ff \subset H$ and $\cap \ff \neq \emptyset$.  Then, there is a subfamily $\ff' \subset \ff$ of at most $d$ sets such that $\cap \ff' \subset H$. 
\end{lemma}

To prove Theorem \ref{thm:hcolorcont}, we will prove the following equivalent formulation.

\begin{theorem}
For any positive integer $d$ and $\ep >0$, there is an $n=n^h(d, \ep)$ such that the following holds.  Let $\ff_1, \ldots, \ff_n$ be $n$ families of closed halfspaces such that for each $i$
\[ \vol\left( \bigcap_{H\in \ff_i} H\right) \le 1.\] 
Then, there is a choice $H_1 \in \ff_1, \ldots , H_n \in \ff_n$ such that
\[ \vol{\left( \bigcap_{i=1}^n H_i\right)} \le 1+\ep. \]
\end{theorem}

\begin{proof}
Let $\ep',\ep''$ be values depending on $\ep$, to be chosen later, and suppose that $n=\Omega(d\cdot n^*(d,\ep''))$.

Applying Theorem \ref{theorem-helly-for-volume} in the contrapositive, we replace each $\ff_i$ by a subset $\ff'_i\subseteq \ff_i$ such that we have $|\ff'_i|\le d\cdot n^*(d,\ep')$ and
\[ \vol\left( \bigcap_{H\in \ff'_i} H\right) \le 1+\ep'.\]

\textbf{Claim.} There exists some choice of $H_1\in \ff'_1,\ldots,H_n\in \ff'_n$ such that $\bigcap_{i=1}^n H_i$ has finite volume.

Observe that translating halfspaces in different directions does not affect whether their intersection has finite volume, though it may affect the value of that volume.  Given $H\in \ff'_i$, we may consider the hyperplane that defines this halfspace; by invariance under translation, we may suppose that all these hyperplanes are tangent to the unit sphere centered at the origin.  Now, applying hyperplane-point duality, each family $\ff'_i$ is transformed to a family of points for which the convex hull contains the origin.  Applying standard colorful Helly's theorem, there must exist a rainbow set of points for which the convex hull contains the origin.  This corresponds to our desired choice $H_1\in \ff'_1,\ldots,H_n\in \ff'_n$, proving the claim.

Now, suppose that $H_1\in \ff'_1,\ldots,H_n\in \ff'_n$ are chosen such that $V=\vol\left( \bigcap_{i=1}^n H_i\right)$ attains the minimum value. Applying Theorem \ref{theorem-helly-for-volume}, again in the contrapositive, there exists some subset $S\subseteq \{1,\ldots,n\}$ such that $|S|\le d\cdot n^*(d,\ep'')$ such that $$\vol\left( \bigcap_{i\in S} H_i\right)\le (1+\ep'')V.$$Let $P$ be the polytope defined by $\{H_i\mid i\in S\}$, and let $j$ be an element of $\{1,\ldots,n\}\backslash S$.  We will attempt to find some $H\in \ff'_j$ that significantly reduces the volume of $P$.

Suppose towards a contradiction that, for each $H$, we have $$\vol(P\cap H)> (1+\ep'')V-\frac{1}{|\ff'_i|}\left[(1+\ep'')V-(1+\ep')\right].$$Then, we would have $$\vol\left( \bigcap_{H\in \ff'_j} H\right)\ge \bigcap_{H\in \ff'_j} \vol(P\cap H)>1+\ep',$$a contradiction. Hence, for some $H$ we must have
\begin{align*}
\vol(P\cap H)&\le (1+\ep'')V-\frac{1}{|\ff'_i|}\left[(1+\ep'')V-(1+\ep')\right]\\
&\le (1+\ep'')V-\frac{(1+\ep'')V-(1+\ep')}{d\cdot n^*(d,\ep')}.
\end{align*}

However, we assumed that the intersection of any colorful set of halfspaces has volume at least $V$.  Hence, $$V\le (1+\ep'')V-\frac{(1+\ep'')V-(1+\ep')}{d\cdot n^*(d,\ep')},$$which rearranges to $$V\le \frac{1+\ep'}{(1+\ep'')-\ep''\cdot d\cdot n^*(d,\ep')}\approx 1+\ep'\cdot \ep''\cdot d\cdot n^*(d,\ep').$$Thus, the theorem holds if we choose $\ep',\ep''$ such that $\ep''\cdot d\cdot n^*(d,\ep')\ll 1$ and $\ep'\cdot \ep''\cdot d\cdot n^*(d,\ep')<\ep$.
\end{proof}


Similar statements hold with diameter as with volume.  Given two convex sets $C, D$, we denote their Hausdorff distance $\delta_{H}(C,D)$; then, we have
\[
|\operatorname{diam}(C)-\operatorname{diam}(D)| \le 2\delta_{H}(C,D).
\]
It is a classic problem to approximate a convex set by a polytope with few facets that contains it and is close in Hausdorff distance \cite{bronstein2008approximation, dudley1974metric, bronsteinivanov}. For any convex set $K$, there is a polytope of $O(\ep^{-(d-1)/2})$ vertices with Hausdorff distance at most $\ep$ (the $O$ notation hides constants depending on $K$).  Thus it makes sense to define the following.

\begin{proposition}
\label{definition-diameter}
Let $d$ be a positive integer and $\ep>0$.  Then, there exists an integer $n$ such that for any convex set $K \subset \R^d$ with positive volume, there is a polytope $P \supset K$ of at most $n$ facets such that
\[
\operatorname{diam}(P) \le (1+\ep) \operatorname{diam}(K).
\]

We define $n^{\operatorname{diam}}(d,\ep)$ to be the smallest such value $n$.
\end{proposition}

\begin{proof}
From the discussion above, if we fix $K$, we know that $n^{\operatorname{diam}}(d,\ep,K)$ exists and is $O(\ep^{-(d-1)/2})$. Fix $\ep$ and $d$.

In order to get a universal bound for $n^{\operatorname{diam}}(d,\ep)$, note that it is sufficient to show the existence for the family $\mathcal{C}$ of closed convex sets $K \subset B_1 (0)$ with diameter $2$.  If there was no upper bound for $n^{\operatorname{diam}}(d,\ep)$, we would be able to find a sequence of convex sets such that $n^{\operatorname{diam}}(d,\ep,K_i) \to \infty$.  Since $\mathcal{C}$ is compact under the Hausdorff topology, there is a convergent subsequence.  If $K_i \to \tilde{K}$, one can see that polytopes that approximate $\tilde{K}$ very well would approximate $K_i$ as well if $i$ is large enough (a small perturbation is needed to fix containment, with arbitrarily small effect on the diameter).   This leads to the fact that $\limsup_{i \to \infty} n^{\operatorname{diam}}(d,\ep,K_i)$ is bounded by $n^{\operatorname{diam}}(d,\ep, \tilde{K})$, a contradiction.
\end{proof}

The fact that we have to work with convex sets up to homothetic copies is the reason why we can get bounds which approximate diameter with a relative error as opposed to an absolute error.

With $n^{\operatorname{diam}}(d,\ep)$ defined, we just need to follow the same proof as we did with Theorem \ref{theorem-helly-for-volume} to obtain the proof of Theorem \ref{theorem-helly-for-diameter}.

\subsection{Discrete quantitative Helly}


We presented the $S$-Helly number in the introduction. To compute
the $S$-Helly number when $S$ is a discrete subset of $\R^d$, it suffices to consider (finite) families of convex polytopes whose vertices are in $S$, instead of families of arbitrary convex sets.
In the work of Hoffman \cite{Hoffman:1979ix} and later Averkov \cite{Ave2013}, the Helly numbers of various sets $S$ were calculated using this approach. Here we extend their work to take into account the \emph{cardinality} of the intersections with $S$.

\begin{definition}
Say that a subset $P$ of $S\subset\R^d$ is \emph{$k$-Hoffman} if 
$$\left|\bigcap_{p\in P}\conv(P\setminus\{p\})\cap S\right|<k.$$
The \emph{quantitative Hoffman number} $\h'_S(k)$ of a set $S\subset \R^d$ is the largest cardinality of a $k$-Hoffman set $P\subseteq S$.
\end{definition}

This generalizes earlier work; in \cite{Hoffman:1979ix}, Hoffman proved that  $\h_S=\h'_S(1)$, where $\h_S$ is the $S$-Helly number we discussed in the introduction. Here we extend this result to $k>1$.

\begin{lemma}
\label{qHoffman} Let $S \subset \R^d$ be a discrete set.
The quantitative Hoffman number $\h'_S(k)$ bounds the quantitative Helly number $\h_S(k)$ as follows:
$\h'_S(k)-k+1 \leq \h_S(k)\leq \h'_S(k)$. 
\end{lemma}

\begin{proof}
%
%
We begin by proving that $\h_S(k)\geq \h'_S(k)-k+1$.
To do so, let $U\subset S$ be some finite set such that $\left|\bigcap_{u\in U}\conv(U\setminus\{u\}) \cap S \right|<k.$
By the definition of $\h'_S(k)$, $|U|\leq \h'_S(k)$.

Consider the family $\mathcal F=\{\conv(U\setminus\{u\})|u\in U\}$.
By definition, $|\bigcap\mathcal F\cap S|<k$.
Note that if $\mathcal F'$ is a subfamily of $\mathcal F$ with cardinality $|U|-k$, then $\mathcal F'=\{\conv(U\setminus\{u\})|u\in U\setminus U'\}$ for some $U'\subseteq U$ of cardinality $k$. Consequently,
$$U'\subseteq\bigcap_{\substack{u\in U\setminus U'}}\conv(U\setminus\{u\})\cap S.$$
Hence the quantitative Helly number 
$\h_S(k)$ must be greater than $|U|-k$.
That is, $\h_S(k)\geq  \h'_S(k)-k+1$.\\

To prove the other inequality, let $K_1,\dots,K_{\h_S(k)}$ be convex sets such that $\left|\bigcap_{j\neq i}K_j\cap S\right|\geq k$ for all $i\in[\h_S(k)]$ (where $[m]=\{1, \dots, m\}$)
yet $\left|\bigcap_{i\in[\h_S(k)]}K_i\cap S\right|<k$. Such a family $\{K_i\}$ exists by the definition of the quantitative Helly number. Then for all indices $i\in[\h_S(k)]$, there exists $U_i\subseteq\bigcap_{i\neq j}K_j\cap S$ with $|U_i|\geq k$.

Suppose $u\in U_i\cap U_j$ (for some $i\neq j$). Then $u\in\bigcap_i K_i\cap S$, 
so there can be no more than $k-1$ such points.
Hence, for each $i\in[\h_S(k)]$, there exists $u_i\in U_i$ such that 
$u_i\notin\bigcup_{j\neq i}U_j$.
In particular, the $u_i$ are distinct. 
Define now $U=\{u_i|i\in[\h_S(k)]\}$.
Consider $\bigcap_{u\in U}\conv(U\setminus\{u\})\cap S$.
Note that $U\setminus\{u_i\} = \bigcup_{i\neq j}\{u_j\}$. Because $u_j\in K_i$ for all $j\neq i$, $U\setminus\{u_i\}\subseteq K_i$.
Therefore
\begin{align*}
\left|\bigcap_{u\in U}\conv(U\setminus\{u\})\cap S\right| & = \left|\bigcap_{i\in[\h_S(k)]}\conv(U\setminus\{u_i\})\cap S\right| \\
& \leq \left|\bigcap_{i\in[\h_S(k)]}K_i\cap S\right| \\
& < k.
\end{align*}
 By the definition of $\h'_S(k)$, it follows that $\h_S(k)\leq \h'_S(k)$.
\end{proof}

The following notion is easier to work with directly than the Hoffman number:
\begin{definition}
A set $P\subset S$ is \emph{$k$-hollow} if 
$$\big|(\conv(P)\setminus V(\conv(P)))\cap S\big|<k,$$ where $V(K)$ is the vertex set of $K$.
\end{definition}

To relate this notion to the Hoffman number, we have the following lemma.
\begin{lemma}
\label{qScarf}
Let $S\subset\R^d$ be a discrete set.
Then every $k$-hollow set is $k$-Hoffman.
Moreover, for every $k$-Hoffman set $P$ there exists a $k$-hollow set with $P'$ with $|P'|=|P|$ and $\conv(P')\subseteq\conv(P)$.
This means $\h'_S(k)$ is equal to the cardinality of the largest $k$-hollow set.
\end{lemma}

Lemma \ref{qScarf} is a partial generalization of Proposition 3 from \cite{Hoffman:1979ix}.

\begin{proof}
Let $P$ be a finite subset of $S:=L\setminus\bigcup_iL_i$.

Suppose $P$ is $k$-hollow, and note that
$$\bigcap_{p\in P}\conv(P\setminus\{p\})\subseteq\conv(P)\setminus V(\conv(P)).$$
Hence
$$\left|\bigcap_{p\in P}\conv(P\setminus\{p\})\cap S\right|
\leq\big|(\conv(P)\setminus V(\conv(P)))\cap S\big|<k.$$
Thus any $k$-hollow set is also $k$-Hoffman.

Suppose $P$ is a $k$-Hoffman set.
If $P$ is not $k$-hollow, then there 
exists a set $K$ of cardinality $k$ in $(\conv(P)\setminus V(\conv(P)))\cap S$.
Because $P$ is $k$-Hoffman, at least one element $q\in K$ cannot be in
$\bigcap_{p\in P}\conv(P\setminus\{p\}\cap S$.
That is, there exists some $p_0$ in $P$ such that $q\notin\conv(P\setminus\{p_0\})$.
Define $P'=\{q\}\cup P\setminus\{p_0\}$.
Clearly $|P'|=|P|$ and $\conv(P')\subsetneq\conv(P)$.
We may apply this procedure repeatedly.
Since $S$ is discrete, it must terminate in some $k$-hollow $P'$ with
$|P'|=|P|$ and $\conv(P')\subseteq\conv(P)$.
\end{proof}


Lemmas \ref{qHoffman} and \ref{qScarf} allow us to bound the quantitative Helly number 
for some interesting discrete subsets of $\R^d$, by 
finding an upper bound on the largest $k$-hollow set.


\begin{proof}[Proof of Theorem \ref{quantitative-discrete-doignon}]
%
%
Let $P$ be a subset of $L\setminus\bigcup_i L_i$ with cardinality $n^r+1$, where $n=k\cdot2^{m+1}+1$.
We will show that $P$ is not $k$-hollow; this implies Theorem \ref{quantitative-discrete-doignon} by lemma \ref{qHoffman} and the contrapositive of lemma \ref{qScarf}.

It is a simple fact, first observed in \cite{rabinowitz}, that there must 
exist $n+1$ collinear points $z_0,z_1,\dots,z_n$ in $\conv(P) \cap L$ with $z_0,z_n\in L\setminus\bigcup_i L_i$.
Note that $z_{j}$ and $z_{j+\ell}$ cannot both be in $L_i$ if $j=0\mod\ell$; otherwise, $z_0$ would be in $L_i$.
Suppose $z_i$ is in some sublattice $L_j$ for all $i\in[\ell\cdot 2^m,(\ell+1)2^m]$.
By the above note, 
$z_{\ell\cdot 2^m+2^a}$ and $z_{\ell\cdot 2^m+2^b}$ cannot both be in the same sublattice for any $0\leq a<b\leq m$.
This is impossible, as there are $m+1$ values of $a$ and only $m$ sublattices $L_i$.
Therefore $z_i\in L\setminus\bigcup_jL_j$ for some $i\in[\ell\cdot 2^m, \ell\cdot 2^m+1,\dots,(\ell+1)\cdot2^m]$.
Since $n=k\cdot 2^{m+1}+1$, 
$$\left\{\left\{z_i\middle|i\in\left[(2\ell-1) 2^m,2\ell\cdot2^m\right]\right\}\middle|1\leq\ell\leq k\right\}$$
is a family of $k$ disjoint subsets of $\{z_1,\dots,z_{n-1}\}$, each of which contains a point in $L\setminus\bigcup_jL_j$.
It follows that $\{z_1,\dots,z_{n-1}\}$ contains at least $k$ elements of $L\setminus\bigcup_jL_j$.
None of these points can be vertices of $\conv(P)$, as they are strictly between the two endpoints.
Hence
$$\left|\left(\conv(P)\setminus V(\conv(P))\right)\cap L\setminus\bigcup_jL_j\right|\geq k.$$
That is, $P$ is not $k$-hollow.
Consequently, no $k$-hollow set $P$ can have size greater than $\left(2^{m+1}k+1\right)^r$.
Therefore $\h_S(k)\leq \h'_S(k)\leq\left(2^{m+1}k+1\right)^r$.
\end{proof}

{\section{Proofs of quantitative Tverberg theorems}\label{section-Tverberg}



\subsection{Continuous quantitative Tverberg}

\begin{proof}[Proof of Theorem \ref{theorem-tverberg-volume}]
First consider the case with $(2dm-1)(d+1)+1$ sets.   Consider $C = \{c_1, c_2, \ldots, c_n\}$ the set of centers of the balls of unit radius defined in the statement of the theorem.  If we use the standard Tverberg's theorem with the set $C$, we can find a partition of $C$ into $2dm$ sets $C_1, C_2, \ldots, C_{2dm}$ such that their convex hulls intersect in some point $p$.
	
Now we split these $2dm$ parts into $m$ blocks of $2d$ parts each in an arbitrary way.  We now show that in each block, we can pick one point of each of its corresponding $T_i$ such that the convex hull of the resulting set contains $B_{r(d)}(p)$, effectively proving the theorem (remember that $r(d) \ge d^{-2(d+1)}$).
	
Without loss of generality, we assume that one such block is $C_1, C_2, \ldots, C_{2d}$.  For each $C_i$ consider
	\[
	\tilde{C}_i = \bigcup\{T_j : c_j \in C_i\}.
	\]
	Since $\conv(T_j) \supset c_j + B_1(0)$, we have
	\[
	\conv(\tilde{C}_i) \supset \conv(C_i) + B_1(0) \supset p + B_1(0) = B_1(p).
	\]
	Thus, we can apply Theorem \ref{theorem-colored-steinitz}, our colorful Steinitz with guaranteed containment of small balls,  to the sets $\tilde{C}_1, \tilde{C}_2, \ldots, \tilde{C}_{2d}$ and obtain a set $\{t_1, t_2, \ldots, t_{2d}\}$ with $t_i \in \tilde{C}_i$ whose convex hull contains $B_{r(d)}(p)$, as desired.
	If we are given instead $n= (m\cdot [(n'-1)d+1]-1)(d+1)+1$, we can split the sets of balls' centers into $m$ blocks of size $(n'-1)d+1$, which allows us to use Proposition \ref{corollary-steinitz-for-balls} to reach the conclusion. 
\end{proof}

\noindent \textbf{Remark.}  The resulting set of the proof above uses only $2dm$ points, so most sets $T_i$ are not being used at all; this suggests that a stronger statement may hold. Moreover, once we get the first Tverberg partition, we have complete freedom on how to split the $2dm$ parts into $m$ block of equal size. Thus, our approach in fact shows that there exist $\sim m^{2dm}$ different Tverberg partitions of this kind.



With essentially the same method we can prove Theorem \ref{theorem-colorful-tverberg-volume}.

\begin{proof}[Proof of Theorem \ref{theorem-colorful-tverberg-volume}]
	Let $c_{i,j}$ be the center of a ball of unit radius contained in $\conv (T_{i,j})$.  Note that we can apply the colorful Tverberg theorem in \cite[Theorem 2.1]{bmz-optimal} to the set of centers to obtain a colorful partition of them into $\lceil 2dm +1\rceil_p-1 \ge 2dm$ sets whose convex hulls intersect.  As in the proof of Theorem \ref{theorem-tverberg-volume}, we may split these sets into $m$ blocks of exactly $2d$ parts each, leaving perhaps some sets unused.
	The same application of Theorem \ref{theorem-colored-steinitz} gives us the desired result.  If we seek a ball of almost the same radius in the end, Corollary \ref{corollary-steinitz-for-balls} completes the proof.
\end{proof}

\subsection{Discrete quantitative Tverberg}
 





We now turn our attention to the discrete quantitative Tverberg theorem, Theorem \ref{thm:quantitative-disc-tverberg}. We will use the notion of the \textit{depth} of a point inside a set to present a cleaner argument.  We say that a point $p$ has \emph{depth} at least $D$ with respect to a set $A$ if for every closed halfspace $H^+$ containing $p$, we have $|H^+\cap A|\ge D$.  We say that a set of points $P$ has \emph{depth} at least $D$ with respect to $A$ if every $p\in P$ has depth at least $D$ with respect to $A$. We will use the following lemma in the proof of Theorem \ref{thm:quantitative-disc-tverberg}.

\begin{lemma}
\label{lem:depth}
If a set of points $P$ has depth at least 1 with respect to $A$, then the convex hull of $A$ contains $P$.
\end{lemma}
\begin{proof}
If this were not true, then there would exist some hyperplane $H$ separating $\conv(A)$ from a point $p\in P$, contradicting the definition of having depth at least one.
\end{proof}





\begin{proof}[Proof of Theorem  \ref{thm:quantitative-disc-tverberg}]
Suppose that $A\subseteq S$ contains $\h_S(k)(m-1)kd+k$ points.  We will construct an $m$-Tverberg partition of $A$. For this, consider the family of convex sets
$$\mathcal F=\left\{ F|F\subset A,|F|=(\h_S(k)-1)(m-1)kd+k\right\}.$$

Note that for any $F\in\mathcal F$, we have $|A\setminus F|=(m-1)kd$.  Therefore, if $\mathcal G$ is a subfamily of $\mathcal F$ with cardinality $\h_S(k)$, we must have
$$\left| A\setminus\bigcap_{G\in \mathcal G} G\right| \le \h_S(k)(m-1)kd.$$
Since there are $\h_S(k)(m-1)kd+k$ points in $A$, $\bigcap_{G \in \mathcal G} G$ must contain at least $k$ elements of $S$.  Hence, by the definition of the quantitative Helly number $\h_S(k)$, $\bigcap_{F\in \mathcal F} F$ contains at least $k$ elements of $S$. Let $P=\{p_1,p_2,\dots,p_k\}$ be $k$ of those points.

\textbf{Claim 1.} The set $P$ has depth at least $(m-1)kd+1$ with respect to $A$.

Suppose that this is not true. Then, some closed halfspace $H^+$ contains an element of $P$ and at most $(m-1)kd$ elements of $A$.  This means that there are at least $(\h_S(k)-1)(m-1)kd+k$ elements of $A$ in the complement of $H^+$. However, this means that some $F\in \mathcal{F}$ lies in the complement of $H^+$, a contradiction, since every such $F$ must contain all points of $P$.

The theorem now follows immediately from the following claim:

\textbf{Claim 2.} For each $j\le m$, we can find $j$ disjoint subsets $A_1,A_2,\dots,A_j \subset A$ such that $P\subset \conv(A_i)$ for every $i$.

We proceed by induction on $j$.  In the base case of $j=1$, Claim 1 tells us that $P$ has depth at least one with respect to $A$.  Hence, by Lemma \ref{lem:depth}, we have $P\subset \conv(A)$.  Now suppose that $j>1$.  By our inductive hypothesis, we can find $j-1$ disjoint subsets $A_1,A_2,\dots,A_{j-1}\subset A$ such that $P\subset \conv(A_i)$ for every $i$.

\textbf{Case 1.} $k\ge 2$.

Applying Theorem \ref{theorem-quantitative-discrete-caratheodory}, we may assume that $A_1,A_2,\ldots,A_{j-1}$ have cardinality at most $kd$, so the depth of $P$ is diminished by at most $kd$ if we remove $A_i$ from $A$.  It follows that the depth of $P$ is at least 1 with respect to $A\setminus \cup_{1\le i\le j-1} A_i$.  Hence, by Lemma \ref{lem:depth}, we can find $A_j\subset A$ disjoint from $A_1,A_2,\ldots,A_{j-1}$ such that $P\subset \conv(A_j)$.

\textbf{Case 2.} $k=1$.

In this case, $P=\{p\}$. By standard Carath\'eodory's theorem, we may assume that each $A_i$ (for $1\le i\le j-1$) either has cardinality less than $d+1$ or else has cardinality $d+1$ and defines a full-dimensional simplex. Notice that every halfspace containing $p$ can contain at most $d$ points of $A_i$, since $p\in \conv(A_i)$. Proceeding as in Case 1, we conclude that $p$ has depth at least 1 with respect to $A\setminus \cup_{1\le i\le j-1} A_i$.  Hence, by Lemma \ref{lem:depth}, we can find $A_j\subset A$ disjoint from $A_1,A_2,\ldots,A_{j-1}$ such that $P\subset \conv(A_j)$.

This completes our induction and proves the theorem.
\end{proof}}


\section*{Acknowledgments} We are grateful to A. Barvinok for his comments and suggestions. 
This work was partially supported by the Institute for Mathematics and its Applications (IMA) in Minneapolis, MN funded by the National Science Foundation (NSF). The authors are grateful for the wonderful working environment that lead to this paper. The research of De Loera and La Haye was also supported by a UC MEXUS grant. Rolnick was additionally supported by NSF grants DMS-1321794 and 1122374.

\bibliographystyle{plain}

\bibliography{references}

\noindent J.A. De Loera and R.N. La Haye\\
\textsc{
Department of Mathematics \\
University of California, Davis \\
Davis, CA 95616
 \\
}\\[0.3cm]
\noindent D. Rolnick \\
\textsc{
 Department of Mathematics \\
 Massachusetts Institute of Technology \\
 Cambridge, MA 02139
 \\
}\\[0.3cm]
\noindent P. Sober\'on \\
\textsc{
Mathematics Department \\
University of Michigan \\
Ann Arbor, MI 48109-1043
}\\[0.3cm]

\noindent \textit{E-mail addresses: }\texttt{deloera@math.ucdavis.edu, rlahaye@math.ucdavis.edu, drolnick@math.mit.edu, psoberon@umich.edu}
\end{document}